\documentclass[a4paper,11pt]{amsart}
\usepackage{amsmath,amsthm,amsfonts,amssymb}
\theoremstyle{plain}
\newtheorem*{thm}{Theorem}

\newtheorem*{que}{Question}

\newcommand{\Ker}{\operatorname{Ker}}

\newcommand{\Irr}{\operatorname{Irr}}

\title{On the number of constituents of \\ products of characters}

\author{Maria Loukaki}
\address{Department of Applied Mathematics\\
University of Crete\\
 Knosou Av. 74 GR-71409 Heraklion-Crete GREECE}
\email{loukaki@gmail.com}

\author{Alexander Moret\'o}
\address{Departament d'\`Algebra \\
Universitat de Val\`encia \\
46100 Burjassot. Val\`encia SPAIN}
 \email{Alexander.Moreto@uv.es}

 \thanks{The second author was supported by the FEDER,  the Spanish
Ministerio de Ciencia y Educaci\'on, grants MTM2004-04665 and
MTM2004-06067-C02-01, the Generalitat Valenciana and the Programa
Ram\'on y Cajal.\\
{\it Key words:} Products of faithful irreducible characters\\
{\it 2000 Mathematics Subject Classification}  Primary 20C15\\
Received by the editors April 12, 2005   and,  in revised form, January 9, 2006}

\begin{document}

\begin{abstract}It has been conjectured (see \cite{ada}) that if the number of
distinct irreducible constituents of the product of two faithful
irreducible characters of a finite $p$-group, for $p\geq5$,  is
bigger than $(p+1)/2$, then it is at least $p$. We give a
counterexample to this conjecture.
\end{abstract}

\maketitle

In this note, $p$ will stand for a prime number. Let $\varphi$ and
$\psi$ be faithful irreducible characters of a finite $p$-group $P$. What can be
said about the number of different irreducible constituents of the
product $\varphi\psi$? At first sight, it does not seem reasonable
to expect strong restrictions for the possible values of this
number. However, in \cite{ada} it was proved that if the number of
constituents of this product is bigger than one, then it is at least
$(p+1)/2$.

It seems reasonable to ask what further restrictions can be found.
In p. 237 of \cite{ada} it was conjectured that if the number of
constituents of the product of two faithful irreducible characters of a finite
$p$-group, for $p\geq5$,  is bigger than $(p+1)/2$, then it is at
least $p$. The theorem that follows shows that this is not true.

\begin{thm}
Let $P=C_p\wr C_p$ for $p\geq5$. There exist
$\varphi,\psi\in\Irr(P)$ faithful such that $\varphi\psi$ has
exactly $p-1$ distinct irreducible constituents.
\end{thm}

\begin{proof}
Write $P=CA$, where $A$ is the base group, which is elementary
abelian of order $p^p$. It is clear that the non-linear characters
of $P$ are induced from characters of $A$. In particular, they
have degree $p$. Fix any non-principal character
$\lambda\in\Irr(C_p)$. Then any character of $A$ can be written in
the form $\nu=\lambda^{i_1}\times\cdots\times\lambda^{i_p}$ for
some integers $i_j=0,\dots,p-1$. Thus, we can identify the
character $\nu$ with the $p$-tuple $(i_1,\dots,i_p)$. It is clear
that $\nu^G\in\Irr(G)$ if and only if not all the $i_j$'s are
equal.

We have that $Z(P)=\{(x,\dots,x)\mid x\in C_p\}$ is the unique
minimal normal subgroup of $P$. Also, if $\nu^G\in\Irr(G)$, it
follows from Lemma 5.11 of \cite{isa} that $\nu^G$ is faithful if
and only if $(x,\dots,x)\not\in\Ker\nu^G$. Notice that if
$\lambda(x)=\varepsilon$ for a primitive $p$th root of unity, then
$$
\nu^G(x,\dots,x)=p\varepsilon^{i_1+\cdots+i_p}.
$$
Thus $\nu^G$ is faithful if and only if
$\sum_{j=1}^pi_j\not\equiv0\pmod{p}$.

Consider the characters of $A$ associated to the $p$-tuples
$(1,0,0,\dots,0)$ and $(1,1,0,\dots,0)$. They induce faithful
irreducible characters of $P$, $\varphi$ and $\psi$ respectively.
We claim that $\varphi\psi$ has $p-1$ distinct irreducible
constituents.

The character $\varphi_A$ decomposes as the sum of the characters
associated to $(1,0,0,\dots,0), (0,1,0,\dots,0),
(0,0,1,\dots,0),\dots,(0,0,0,\dots,1)$. We can argue similarly
with $\psi_A$. The product of two characters of $A$ corresponds to
the componentwise sum of the associated $p$-tuples and two
characters of $A$ are $P$-conjugate if and only we can go from the
$p$-tuple associated to one of the characters to the other by a
cyclic permutation of the components. Now, it is easy to see that
the number of constituents of the character $\varphi\psi$ is the
number of characters of $P$ lying over the characters of $A$
corresponding to the $p$-tuples $(2,1,0,0,\dots,0),
(1,2,0,0,\cdots,0), (1,1,1,0,\dots,0), (1,1,0,1,\dots,0),\dots,\\
 (1,1,0,0,\dots,1)$. Here we have $p$ different $p$-tuples. The
third of these $p$-tuples and the last one correspond to
$P$-conjugate characters of $A$, so they induce the same character
of $P$. It is easy to see that no other pair of $p$-tuples are
conjugate. The claim follows.
\end{proof}

We have been unable to find any example where the number of
constituents of the product of two faithful irreducible characters of a
$p$-group has more than $(p+1)/2$ distinct irreducible
constituents but less than $p-1$. So the following modification of
the conjecture could still be true.

\begin{que}
Let $\varphi$ and $\psi$ be faithful irreducible characters of a
finite $p$-group $P$. Assume that $\varphi\psi$ has more than
$(p+1)/2$ distinct irreducible constituents. Does it necessarily
have at least $p-1$ irreducible constituents?
\end{que}


\begin{thebibliography}{99}


\bibitem{ada}
E. Adan-Bante, Products of characters and finite $p$-groups, {\it
J. Algebra} {\bf 277} (2004), 236--255.

\bibitem{isa}
M. Isaacs, {\sl Character Theory of Finite Groups},
Dover, New York, 1994.


\end{thebibliography}
\end{document}